\documentclass{article}
\usepackage{amssymb,amsfonts, amsmath,amsthm}

\usepackage{color}
\usepackage{graphics}
\usepackage{epsfig}

\newcommand\supp{\mathop{\rm supp}\nolimits}

\newtheorem{thm}{Theorem}[section]

\newtheorem{lem}[thm]{Lemma}

\newtheorem{exam}[thm]{Example}
\newtheorem{defin}[thm]{Definition}

\newcommand\beq{\begin{equation}}
\newcommand\eeq{\end{equation}}

\newcommand\DD{{\mathbb D}}

\newcommand{\TT}{{\mathbb T}}
\newcommand{\G}{{\mathcal G}}
\newcommand\ZZ{\mathbb Z}

\title{The group of invariants of an inner function with finite spectrum}

\author{Isabelle Chalendar\\ \small{Universit\'e de Lyon; CNRS; Universit\'e Lyon 1; INSA de Lyon; Ecole Centrale de Lyon}\\ \small{ CNRS, UMR 5208, Institut Camille Jordan}\\ \small{ 43 bld. du 11 novembre 1918, F-69622 Villeurbanne Cedex, France}\\ \small{\tt E-mail: chalenda@math.univ-lyon1.fr}
\and Pamela Gorkin\\ \small{Department of Mathematics, Bucknell University, Lewisburg, PA 17837, U.S.A.} \\ \small{\tt E-mail: pgorkin@bucknell.edu} 
\and and \\ Jonathan R. Partington\\\small{School of Mathematics, University of Leeds, Leeds LS2 9JT, U.K.} \\ \small{\tt E-mail: J.R.Partington@leeds.ac.uk}}

\begin{document}

\maketitle

\begin{abstract}
This paper determines  the group of continuous invariants corresponding to an inner function $\Theta$ with finitely many singularities
on the unit circle $\TT$;
that is, the continuous mappings $g: \TT \to \TT$ such that $\Theta \circ g = \Theta$ on   $\TT$. These mappings form
a group under composition.
\end{abstract}

\section{Introduction}

In this paper, we study the group $\G$ of continuous invariants corresponding to an inner function $\Theta$ with finitely many singularities
on the unit circle $\TT$;
that is, the continuous mappings $x: \TT \to \TT$ such that $\Theta \circ x = \Theta$ on   $\TT$.
Since $\Theta$ is undefined at its singular points on $\TT$, we interpret this as meaning that
$\Theta$ maps singular points to singular points and regular points to regular points.
These mappings form
a group under composition. 

In \cite{CC}, it was shown that the group $\G$ of invariants of a finite Blaschke product $B$ of degree $n$  is isomorphic to the cyclic group $\ZZ_n$.
In this case the circle may be divided into $n$ sub-arcs on each of which $B$ takes all values of $\TT$ precisely once; then the
mappings $x \in \G$ permute the arcs cyclically. The study of $\G$
was useful in a problem of dual algebra theory \cite[Thm. 3.1]{CCC}. 
 When we come to study general inner functions, the
situation is considerably more difficult, but it is still possible to give a 
 complete solution
in the case that the spectrum is finite; we shall see that various complicated groups
can arise, in particular infinite and non-Abelian groups.

Note that if $x \in \G$, then at a regular point of $\TT$ (that is, a point that is not a singularity of $\Theta$), we may use the equality  $\Theta \circ x = \Theta$ and equation (\ref{eqn:derivative}) below to conclude that 
$x$ is analytic on a neighbourhood of each regular point
with $\arg x(z)$ strictly increasing with $\arg z$.

\section{Notation}
We refer to \cite{hoffman} for standard results on Hardy spaces and inner functions.
Recall that an inner function $\Theta$ is a bounded analytic function in $\DD$ whose boundary values 
satisfy $|\Theta(z)|=1$ a.e., for $z$ on the unit circle  $\TT$; such a function may be factorized into the product of
a (finite or infinite) Blaschke product and a singular inner function.

An infinite Blaschke product defined on the disc $\DD$ is an analytic function $B$
of the form 
\[
B(z) = \lambda z^p  \prod_{j = 1}^\infty \frac{|a_j|}{a_j}\frac{a_j - z}{1 - \overline{a_j}z},
\] where $a_j \in \mathbb{D}$, $|\lambda| = 1$, 
$p$ is a non-negative integer and $\sum_{j=1}^\infty (1-|a_j|) < \infty$ (a finite Blaschke
product is defined analogously).

A singular inner function may be written as
\[
S(z)=\exp \left[ -\int \frac{e^{it}+z}{e^{it}-z} \, d\mu(t) 
\right],
\]
where $\mu$ is a singular positive measure on $\TT$.

The spectrum $\sigma(\Theta)$ is the complement of the set of points $p \in \TT$ such that $\Theta$
has an analytic extension into a neighbourhood of $p$. Indeed, writing $Z(\Theta)$ for the zero set of
 $\Theta$ in $\DD$ we have
\[
\sigma(\Theta)= \TT \cap \left\{ \overline{Z(\Theta)} \cup \supp(\mu) \right\}.
\]

\section{Background}

Consider an inner function $\Theta$ with finitely many points in the spectrum, $\sigma(\Theta)$. It is well known that the essential range of $\Theta$ at a singularity $\xi_0$ is the unit circle, $\partial \mathbb{D}$. For the purposes of this paper, it will be necessary to classify the possible types of limiting behaviour of $\Theta$ at the point $\xi_0$. We do so in the next definition.

\begin{defin}
Let $\Theta$ be an inner function with finite spectrum. Let $\xi_0 \in \sigma(\Theta) \cap \partial \mathbb{D}$. We say that

\begin{enumerate}
\item {\bf $\xi_0 = e^{i \theta_0}$ is of type $1_{a, L}$} if for $\varepsilon > 0$ sufficiently small, there are infinitely many solutions of $\Theta(\xi) = 1$ in the open interval $(e^{i \theta_0}, e^{i (\theta_0 + \varepsilon)})$, finitely many solutions in $(e^{i (\theta_0 - \varepsilon)}, e^{i \theta_0})$, and $\lim_{\alpha \to \theta_0^-} \Theta(e^{i \alpha}) = L$;

\item {\bf $\xi_0 = e^{i \theta_0}$ is of type $1_{b, L}$} if for $\varepsilon > 0$ sufficiently small, there are infinitely many solutions of $\Theta(\xi) = 1$ in the open interval $(e^{i (\theta_0 - \varepsilon)}, e^{i \theta_0})$, finitely many solutions in $(e^{i \theta_0}, e^{i (\theta_0 + \varepsilon)})$, and $\lim_{\alpha \to \theta_0^+} \Theta(e^{i \alpha}) = L$;

\item {\bf $\xi_0 = e^{i \theta_0}$ is of type $2$} if there are infinitely many solutions to $\Theta(\xi) = 1$ in both of the intervals $(e^{i (\theta_0 - \varepsilon)}, e^{i \theta_0})$ and $(e^{i \theta_0}, e^{i (\theta_0 + \varepsilon)})$.

\end{enumerate}

\end{defin}

We remark that there is nothing special about the value $1$ above; that is, if $B$ has only finitely many solutions to $B(e^{i \theta}) = 1$ in an interval, it has only finitely many solutions to $B(e^{i \theta}) = \lambda$ for any $\lambda \in \partial \mathbb{D}$.

We will soon show that these three intervals exhaust all possibilities and we will give examples of 
Blaschke products satisfying each possible situation. It will, however, be convenient to have a way to refer to 
the corresponding intervals of each type. Thus, we have the following.

\begin{defin}
Given an inner function $\Theta$ an interval $(e^{i \theta_0}, e^{i \theta_1})$

\begin{enumerate}
\item {\bf is of type $0$ (for $\Theta$)} if there are only finitely many solutions to $\Theta(\xi) = 1$ in the interval;
\item {\bf is of type $1_a$ (for $\Theta$)} if there exist infinitely many solutions to $\Theta(\xi) = 1$ in the interval accumulating precisely at the point $e^{i \theta_0}$;
\item {\bf is of type $1_b$ (for $\Theta$)} if there exist infinitely many solutions to $\Theta(\xi) = 1$ in the interval accumulating precisely at the point $e^{i \theta_1}$;
\item {\bf is of type $2$ (for $\Theta$)} if there are infinitely many solutions to $\Theta(\xi) = 1$ in the interval accumulating precisely  at 
both $e^{i \theta_0}$ and $e^{i \theta_1}$.
\end{enumerate}
\end{defin}

It should be clear that the definition describes all possibilities for the behavior of an inner function with isolated singularities, 
but it may not be clear that each situation actually occurs. 
Before we discuss this further, we note that a theorem of R. Berman \cite[Theorem 4.8]{Berman} 
shows that given a nonempty set $E \subset \TT$ with measure zero and of type $F_\sigma$ and $G_\delta$, there exists a Blaschke product $B$ such that $B^{-1}(1) = E$. 
Thus, there are Blaschke products with singularities of each type. 
For discrete singular inner functions, E. Decker has an interesting related result \cite[Lemma 6]{Decker}.

Writing our inner function $\Theta$ as $\Theta = B S$, where $B$ is a Blaschke product and $S$ is a singular inner function, we discuss the situation in which the zeros of the Blaschke product approach an isolated singularity through a Stolz region at the point $e^{i \theta_0}$; that is, the points are in a region of the form
$$\Gamma_{\alpha} = \{z \in \mathbb{D} : |e^{i \theta_0} - z| < \alpha(1 - |z|)\},$$ where $\alpha > 1$, as well as when the singular inner factor has an isolated singularity at $e^{i \theta_0}$. 


\begin{lem} \label{lemma:factorization}
Let $\Theta = BS$ be an inner function with finitely many singularities and let $e^{i \theta_0} \in \partial \mathbb{D}$ be a point for which there exists a sequence of zeros of $\Theta$ converging to $e^{i \theta_0}$ in a Stolz angle or $e^{i \theta_0} \in \textup{supp} (\mu)$, where $\mu$ is the singular measure corresponding to $S$. Then $e^{i \theta_0}$ is of type $2$.
\end{lem}

\begin{proof} Let $(a_n)$ denote the zero sequence of the Blaschke product $B$. Our argument will apply to points lying on either side of the singularity we consider, and therefore we consider only one side of the point. The other will follow in exactly the same way. Under our hypotheses, we know that either $B$ has a singularity at the point $e^{i \theta_0}$, which we assume without loss of generality to be the point $1$, or $S$ has a singularity at the point $1$, or both. 

Thus, we may (without loss of generality) consider two cases: First, the case in which there are infinitely many $a_n$  in a Stolz angle, $\Gamma_\alpha$, at the point $1$, and second, the case in which the measure corresponding to $S$ has a singularity at the point $1$. Note that since the singularity of $S$ is isolated, we may assume that the measure corresponding to $S$ has an atom at the point $1$. 

So suppose that $B$ has a singularity at the point $1$ and infinitely many $a_n$ lie in the Stolz angle. We may consider only those $a_n$ that also lie in the right half-plane.

Note that on the unit circle at a point $e^{i \theta}$ at which $B$ is analytic, we have
\begin{eqnarray}\label{eqn:derivative}\frac{B^\prime(e^{i \theta})}{B(e^{i \theta})} = \sum_{n} \frac{1 - |a_n|^2}{(e^{i \theta} - a_n)(1 - \overline{a_n} e^{i \theta})}.\end{eqnarray} 

 Thus, if we let $\varepsilon$ be a small positive constant and $\gamma$ denote the arc of the circle from $1$ to $e^{i \varepsilon}$, we have
\begin{eqnarray}\label{derivative}\left|(1/2\pi) \int_\gamma \frac{B^\prime(z)}{B(z)} dz\right| =\left| \int_{0}^{\varepsilon} \frac{ B^\prime(e^{i \theta})}{e^{-i \theta} B(e^{i \theta})} d\theta/2\pi\right| = \int_{0}^{\varepsilon} \sum_n \frac{1 - |a_n|^2}{|e^{i \theta} - a_n|^2} d\theta/2\pi.\end{eqnarray} Now \cite[Exercise 3, p. 41]{Garnett}, $$\int_{0}^{\varepsilon} P_{a_n}(\theta) d\theta/2\pi = \alpha_n/\pi - \varepsilon/2\pi,$$ where $P_{a_n}(\theta) = \frac{1 - |a_n|^2}{|e^{i \theta} - a_n|^2}$ is the Poisson kernel and $$\alpha_n = \arg((e^{i \varepsilon} - a_n)/(1 - a_n)).$$ 
This is illustrated in Figure \ref{fig:garnett}.

\begin{figure}[htbp]
  \begin{center}
    \leavevmode
    {\centering\input{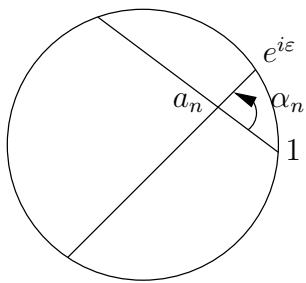}}
    \caption{\emph{The angle $\alpha_n$}}
    \label{fig:garnett}
  \end{center}
\end{figure}

If we choose $a_n$ sufficiently close to $1$, the fact that $a_n$ are all in a Stolz angle implies that there exists $\beta > 0$ such that $\alpha_n \ge (\beta + \varepsilon)/2$ for all such $n$.  Therefore, the change in the argument of $B$ is infinite, and there are infinitely many intervals approaching the point $1$ from above that $B$ wraps around the unit circle. 

Now consider the case in which the singular inner function has an isolated singularity at the point $1$. 
As noted above, if $S$ has finitely many singularities including a singularity at the point $1$, 
then the measure $\mu$ corresponding to $S$ has an atom at the point $1$; that is, 
$\mu = \sum_{k = 1}^m \alpha_k \delta_k$ where $\delta_k$ is point mass at the singularity $e^{i \theta_k}$, 
and $\delta_1$ is point mass at the point $1$. 

In this case, we may write $S = S_1 S_2$, where $S_2$ is a singular inner function continuous at the point $1$ and 
therefore has limit $L$ at the point $1$, where $|L| = 1$. 
Thus, the behavior of $S$ is determined by $S_1$ where $S_1 = (S_{\delta_1})^{\alpha_1} = e^{-\alpha_1 \frac{1 + z}{1 - z}}$, 
where $\alpha_1 > 0$.  Now it is well known (see \cite[Thm. VII.7.48]{Zygmund}) that every (nonconstant) singular inner function 
$S_{\mu}$ has the property that for each $\lambda \in \partial \mathbb{D}$ there exist infinitely many points $e^{i \theta_\lambda}$ 
for which the radial limit of $S_\mu$ satisfies $S_\mu^\star(e^{i \theta_\lambda}) = \lambda$. Since $S$ will assume every value on the unit circle between these points, we see that they must cluster at the singularity. Thus, $S$ assumes the value $\lambda$ on a sequence converging to the point $1$. 
Symmetry considerations show that, in fact, the point $1$ must be of type $2$.

The final case is that in which both functions are discontinuous at the point $1$. In this case, the argument of 
$\Theta$ is the sum of the arguments of the two functions, and the result follows from the paragraphs above.

\end{proof}

\begin{exam}
Given $\lambda \in \partial \mathbb{D}$, there exists a Blaschke product with a singularity of type $1_a$ on the unit circle.  \end{exam}

Suppose $1$ is a singularity of the Blaschke product $B$. Choose a sequence of points  $(a_n)$ in $\mathbb{D}$ approaching $1$ tangentially from above in such a way that the angular derivative at the point $1$ remains bounded. Any sequence $(a_n)$ with $\lim_n a_n = 1$ and $$\sum \frac{1 - |a_n|^2}{|1 - a_n|^2} < \infty$$ will work. For example, we may take $a_n$ so that $n^2(1 - |a_n|^2) < |1 - a_n|^2$, and $a_n$ lies in the upper-half disk. Then,  for $-\varepsilon < t \le 0$, we see that  $|e^{i t} - a_n| > |1 - a_n|$ so
 $$\sum \frac{1 - |a_n|^2}{|e^{i t} - a_n|^2} < \sum \frac{1 - |a_n|^2}{|1 - a_n|^2} < \infty.$$ Therefore the angular derivative is finite at $1$ and bounded on the interval $(e^{i t}, 1)$.  From the version of equation (\ref{derivative}) for the arc below the point $1$, we conclude that the change in argument is bounded on this interval and therefore $B(z) = 1$ can have at most finitely many solutions on this interval. Since $1$ is assumed to be a singularity, there must be infinitely many solutions converging to $1$ from above. Therefore $1$ is a point of type $1_a$.
 
Such examples have appeared in other contexts.  One can be obtained, for example, from the aforementioned work of Berman, or from the expression following (2.1) in Lemma 1 of Leung and Linden \cite{LeungLinden}, since we can construct $b$ so that the radial limit exists at $1$. Then the estimate provided there shows that $\prod_{n = 1}^N b(z, a_n)$ must be close to this limit for $z = r(\theta) e^{i \theta}$ where $r(\theta)$ is sufficiently large and $0 < \theta < \varphi_0$. 
 
 \begin{lem} If $\Theta$ is an inner function with finitely many singularities, then there exists a Blaschke product with precisely the same singularities as $\theta$ of precisely the same type. \end{lem}
 
\begin{proof} This is a consequence of Frostman's theorem \cite[p. 79]{Garnett}, for we may choose $\varphi_a(z) =  \frac{z - a}{1 - \overline{a} z}$  where $\varphi_a$ preserves the order of the points such that $\varphi_a \circ \Theta = B$ is Blaschke. The fact that $\varphi_a$ is a continuous one-to-one mapping of the unit circle onto itself implies that the singularities remain at the same points and of the same type. \end{proof}
 
 Therefore, it remains to analyze the group of invariants for a Blaschke product. The general inner function will then follow in the same way.

\section{Preliminary Observations}

In this section, we continue to consider a Blaschke product $B$ with finitely many singularities on the unit circle. 
We will say that $x \in \G$ {\bf shifts the elements} of an interval 
$I$ if $x(z_{n, \lambda}) = z_{n + 1, \lambda}$ for all $n$ in $\mathbb{Z}$, $\mathbb{Z}^+$
or $\ZZ^-$ (depending on whether the intervals are of type $1$ or type $2$) and all $\lambda \in \partial \mathbb{D}$, 
where $z_{n, \lambda}$ are the solutions to $B(z) = \lambda$ in $I$ ordered with increasing argument.

\begin{lem}\label{lemma:bijective} Let $B$ be a Blaschke product with finitely many singularities and let $x \in \G$. Then $x$ is a bijective mapping of the circle onto itself.
\end{lem}

\begin{proof} We have already observed that $x$
maps  singular points to singular points and regular points to regular points. On each interval between two singular points the argument of $B$ is
strictly increasing (note that $B$ cannot be constant on a set of positive measure), and thus this interval is mapped bijectively under $x$
onto the whole of another such interval. By the continuity of $x$, successive intervals are mapped to successive intervals and since there are only
finitely many, this establishes that $x$ is a bijection of $\TT$.
\end{proof}

\begin{lem}\label{lemma:type2intervals} Let $I$ denote a type $2$ interval and $x \in \G$ such that $x(I)=I$. If there
exists $z \in I$ such that $x(z)=z$, then $x(w)=w$ for all $w \in I$. Hence if 
$x_1, x_2 \in \G$ with 
$x_j(I) = I$ for each $j = 1, 2$, and if there exists $z \in I$ with $x_1(z) = x_2(z)$, then $x_1(w) = x_2(w)$ for all $w \in I$. 
\end{lem}

\begin{proof} 
Suppose, without loss of generality, that $B(z) =1$ and consider the two-sided sequence of points $z_{n, 1} \in I$ such that $B(z_{n, 1}) = 1$,
where $z_{0,1}=z$. These divide $I$ into open subarcs, $(z_{n, 1}, z_{n + 1, 1})$.
Now, being continuous and with strictly increasing argument on each such subarc, $B$ takes all possible values in $\TT$ except 1 exactly once. Moreover,
$x$ fixes $z_{0,1}$, and we now see that $x$, being bijective (see Lemma~\ref{lemma:bijective}), maps each $z_{n,1}$ to itself and fixes the intermediate intervals.
Clearly, the relation $B \circ x=B$ now implies that $x$ fixes all points of $I$.
Finally, if $x_1(z)=x_2(z)$, then $x_1^{-1}x_2$ fixes $I$ and $z$, so fixes each point of $I$.
 \end{proof}

\begin{lem}\label{lemma:type1intervals} 
Let $I$ denote a type $1$ interval or type $0$ interval. If $x \in \G$ and $x(I) = I$, then $x(z) = z$ for all $z \in I$. \end{lem}

\begin{proof} We prove the case for type $1_b$ intervals; the other cases are similar.

We may label the solutions in $I$ to $B(z)=1$ in increasing order of argument as $z_{n,1}$ for $n \ge 0$ (only).
Now, $x$ can only act as a shift on these solutions, but since $x$ is invertible it must fix the first point $z_{0,1}$ as otherwise
either $x$ or $x^{-1}$ would have to map $z_{0,1}$ to a point in $I$ of strictly smaller argument. Now considerations as in the
proof of  Lemma~\ref{lemma:type2intervals} show that $x$ fixes every point of $I$. 

\end{proof}

As a consequence we have:

\begin{lem}\label{lemma:allintervals} Let $x \in \G$. Then $x$ maps intervals of type $0$, $1_a$, $1_b$ and $2$
 to intervals of exactly the same type. Moreover, in the case when an endpoint is of type $1_{a,L}$ or $1_{b,L}$, it is
mapped under $x$ to another endpoint of exactly the same type.
\end{lem}

\begin{lem}\label{lemma:interchangingintervals} Let $I$ and $J$ be intervals of the same type and 
suppose that $y_1, y_2 \in \G$ map $I$ to $J$. If $y_1(z) = y_2(z)$ for some $z$ in $I$, then $y_1 = y_2$ on $I$. \end{lem}

\begin{proof} In this case $y_2^{-1}y_1$ maps $I$ to $I$ and fixes a point of $I$; hence it fixes all points of $I$
by Lemmas \ref{lemma:type2intervals} 
and \ref{lemma:type1intervals} .
\end{proof}

We note that type $0$ intervals can only be mapped to each other if their images under $B$ are identical (including multiplicities of points).

   \section{1, 2 and $p$ singularities: a detailed description}
 
 Before analysing the general case, we consider the cases when the number of singularities
is 1 or 2, and then an arbitrary prime number. Here the number of possible symmetry
groups is very limited, and it is possible to give detailed descriptions of the corresponding groups.
 
 \subsection{One singularity}
 
 \begin{thm} If $B$ is a Blaschke product with a singularity at the point $e^{i \theta_0}$ only, then we have the following:

 \begin{enumerate}
 \item If $e^{i \theta_0}$ is of type $2$, then $\G = \mathbb{Z}$;
 \item If $e^{i \theta_0}$ is of type $1_a$ or $1_b$, the group is the trivial group; that is, $\G = \{e\}$.
 \end{enumerate}
 \end{thm}
 

\begin{figure}[htbp]
  \begin{center}
    \leavevmode
    {\centering\input{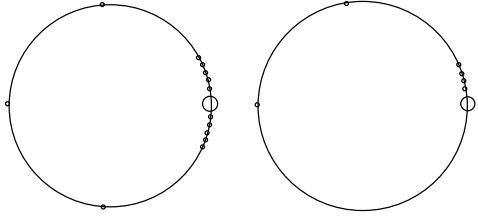}}
    \caption{\emph{Singularities of type $2$, and type $1_a$}}
    \label{fig:page8}
  \end{center}
\end{figure}

\begin{proof} We refer to Figure \ref{fig:page8}
at this point.
Without loss of generality, we assume that $e^{i \theta_0} = 1$.  

First suppose that $1$ is a point of type $2$.
Note that $B$ is continuously differentiable with nonvanishing derivative on the interval $I:=\TT\setminus\{1\}$, and for each $\lambda \in \TT$ we may enumerate
the points such that $B(z)=\lambda$ as a sequence
$\{z_{n,\lambda}\}_{n \in \ZZ}$, taken with increasing argument, such that all points $z_{0,\lambda}$ lie in the arc
between $z_{0,1}$ and $z_{1,1}$. Moreover, for each $\lambda \in \TT$ we have $z_{n,\lambda} \to 1$ as $n \to \pm \infty$.

We may define a mapping $x_0 \in \G$ by $x_0(z_{n,\lambda})=z_{n+1,\lambda}$ for each $n \in \ZZ$ and each $\lambda \in \partial \mathbb{D}$,
and $x_0(1)=1$.
 From our definition, it is clear that  the two functions $B \circ x_0$ and $B$ are equal; also
$x_0$ is (in argument) a monotonic increasing  bijection from $I$ to itself, and hence is continuous.

Now $x_0$ generates the entire group $\G$, for if $y \in \G$ then 
$y: z_{0,1} \mapsto z_{n,1}$ for some $n \in \ZZ$, and then $y=x_0^n$ by Lemma \ref{lemma:type2intervals}.

If $1$ is a point of type $1$, then necessarily $I$ is an interval of type 1, and the group is trivial, by
Lemma~\ref{lemma:type1intervals}.
\end{proof}

The complications begin with the next case.
 
 \subsection{Two singularities}

\begin{thm}
Suppose that $B$ is a Blaschke product with spectrum consisting of exactly two points of $\TT$. Then, depending on
the nature of the two intervals comprising $\TT \setminus \sigma(B)$, the group $\G$ is one of the following:
$\{e\}$, $\ZZ_2$, $\ZZ$, $\ZZ^2 \rtimes \ZZ_2$. The last of these cases is a semi-direct product, which may be presented
as 
\[
\G = \langle x_1, x_2, y: y^2 = e, x_1 x_2 = x_2 x_1, x_2 y = y x_1 \rangle.
\]
\end{thm}

We refer to \cite{algebra}, for example, for more on semi-direct products. Recall that if $\G \cong K \rtimes \mathcal{S}$, then
$\G$ has a normal subgroup isomorphic to $K$ with quotient $\G/K \cong \mathcal{S}$. These do not determine $\G$ uniquely,
but a presentation in terms of generators and relations can be given.
\begin{proof}
 
We may suppose, without loss of generality, that $B$ is a Blaschke product with singularities at $1$ and $-1$.  
Call the resulting intervals $I_1$ and $I_2$, where $I_1$ denotes the upper-half of the unit circle. 

Note first that $\G$ contains a normal subgroup 
\[
K:= \{x \in \G: \, x(I_k)=I_k \hbox{ for each } k \}.
\]
Indeed, we may define a homomorphism $\psi: \G \to \ZZ_2=\{[0],[1]\}$ by defining $\psi(x)=[r]$ if $x(I_k)=I_{k+r}$
for each $k$ (here intervals are numbered modulo 2). Clearly any $x \in \G$ either fixes both intervals $I_k$ or it
interchanges them. Now $\G/K \cong \psi(\G)$, which is either $\{e\}$ or $\ZZ_2$. We consider three cases.
 
 \begin{enumerate}
 \item[Case 1.] Both intervals are of type $2$.
 \item[Case 2.] One interval is type $1$ (or type $0$) and one interval is type $2$.
 \item[Case 3.] Both intervals are type $1$.
 \end{enumerate}

 {\bf Case 1.} (Both intervals type $2$.) From Lemma~\ref{lemma:type2intervals}, we know that there exist $x_j \in \G$ such that $x_j$ 
shifts the elements of $I_j$ (from $z_{n, \lambda}$ to $z_{n + 1, \lambda}$) for $j = 1, 2$ and each $x_j$ leaves 
the elements of the remaining interval fixed. Again as in the case of one interval, we may argue that 
$x_j$ is continuous and it follows that $x_j(1) = 1$ and $x_j(-1) = -1$. As above, we see from this that 
$B(x_j(e^{i \theta})) = B(e^{i \theta})$ for $\theta \ne 0, \pi$.
 
Now consider the point $i$. Since $i$ is not a singularity, the radial limit of $B$ exists at $i \in I_1$ and we may write $B(i) = \lambda$. 
Choose a point, $z_{-} \in I_2$ such that $B(z_{-}) = \lambda$. 
We claim that there is a unique element $y \in \G$ such that $y$ interchanges $i$ and $z_-$: for on  
$I_1$ we may define $y(x_1^n(i))=x_2^n (z_-)$, and define the image under $y$ of the remaining subintervals of $I_1$ in the unique way that
guarantees that $y$ is continuous and $B \circ y=B$. Similarly, we define $y(x_2^n (z_-))=x_1^n( i)$, and fill in the values on $I_2$
in the same way. Note that $y^2=e$.

We claim that \[
\G = \langle x_1, x_2, y: y^2 = e, x_1 x_2 = x_2 x_1, x_2 y = y x_1 \rangle.
\]
Our argument above shows that the right-hand side is a subset of $\G$. Now choose an element $x \in \G$. 

Note that we have various possibilities:

\begin{enumerate}
\item[$\bullet$] $x = e$;
\item[$\bullet$] $x|I_1 = e$, $x|I_2 \ne e$. In this case $x = x_2^n$ for some $n \in \mathbb{Z}$. The case $x|I_1 \ne e$, $x|I_2 = e$ is handled in the same way.
\item[$\bullet$] $x(I_j) = I_j$ and $x|I_j \ne e$ for $j = 1, 2$. Then $x = x_1^n \circ x_2^m$ for $n, m \in \mathbb{Z} \setminus \{0\}$. 
\item[$\bullet$] $x(I_1) = I_2$. In this case $x \circ y(I_1) = I_1$ and we obtain the result from the previous cases. 
\end{enumerate}

This completes the proof in which we have two intervals of type $2$.

\bigskip

{\bf Case 2.}  (One interval type $2$, one not.) In this case one interval, say $I_1$, is type $1$ (or type $0$) and $I_2$ is type $2$. 
By Lemma~\ref{lemma:interchangingintervals} we cannot map $I_1$ to $I_2$, and therefore by 
Lemma~\ref{lemma:type1intervals} every element of $\G$ must fix the points of $I_1$. 
Let $x_0$ denote the shift of elements in $I_2$. Then $$G = \langle x_0 \rangle \cong \ZZ.$$

{\bf Case 3.} (Neither interval type $2$.)  In this case, $B$ has a singularity at $1$ and $-1$. If there were a type $0$ interval, one interval would have to be type $2$. Thus $I_1$ and $I_2$ are type $1$ intervals.  Further, since there are no type $2$ intervals, each singularity is a point of the same type. Thus, the Blaschke product has one-sided limits at each singularity. We have several cases:

\smallskip
{\bf Case 3a: The one-sided limits exist and are not equal.} Without loss of generality, we assume that $\lim_{\theta \to \pi^+} B(e^{i \theta}) \ne \lim_{\theta \to 0^+} B(e^{i\theta})$. 
 By Lemma~\ref{lemma:type1intervals}, the only maps $x \in \G$ for 
which $x(I_j) = I_j$ are the identity. The only other possibility is that there exists $y \in G$ such that $y$ 
interchanges $I_1$ and $I_2$ (see Figure \ref{fig:page10}). In this case, since the points where $B$ assumes the value $1$ are mapped to each other, 
we must map $-1$ to $1$. However, for $\theta \ne 0, \pi$ we have $B(y(e^{i \theta})) = B(e^{i \theta})$ with different
one-sided limits at $1$ and $-1$, and $y$ is continuous at $1$ and $-1$, so this is impossible and the group must be the trivial group.

\begin{figure}[htbp]
  \begin{center}
    \leavevmode
    {\centering\input{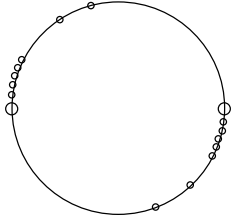}}
    \caption{\emph{The case where $y$ interchanges $I_1$ and $I_2$}}
    \label{fig:page10}
  \end{center}
\end{figure}

\smallskip
{\bf Case 3b: The one-sided limits exist and are equal.} Again, we assume without loss of generality that $\lim_{\theta \to \pi^+} B(e^{i\theta}) = \lim_{\theta \to 0^+} B(e^{i \theta})$. In this case, there is an element $y$ of order $2$ that maps the sequence of points $\{z_{n, \lambda, j}\} \in I_j$ for which $B(z_{n, \lambda, j}) = \lambda$ to the corresponding points in the other interval. Our map is $y(z_{n, \lambda, 1}) = z_{n, \lambda, 2}$ and our group is $\mathbb{Z}_2$.

\end{proof}

\subsection{Prime number of singularities}

In what follows, we let $K = \{y \in \G: y(I_k) = I_k~\mbox{for all}~ k\}$. Define $\psi: \G \to \mathbb{Z}_n$ as follows:
for $y \in \G$  such that $y(I_j) = I_{j + r} \pmod n$, let $\psi(y) = [r]$.  Note that  $K = \ker \psi$, $K \cong \mathbb{Z}^k$ 
for some $k$ with $0 \le k \le n$ and that $k$ is the number of intervals of type $2$.

\begin{thm}\label{prime}
Suppose that $B$ is a Blaschke product with spectrum consisting of exactly  $p$ points of $\TT$, where $p \ge 2$ is
a prime number. Then, depending on
the nature of the  intervals comprising $\TT \setminus \sigma(B)$, the group $\G$ is one of the following:
$\{e\}$, $\ZZ_p$, $\ZZ^r$ for some $1 \le r < p$, $\ZZ^p \rtimes \ZZ_p$. The last of these cases is a semi-direct product, which may be presented
as 
\[
\G = \langle x_1, x_2 \ldots, x_p, y: y^p = e,  x_j x_k = x_k x_j, ~\mbox{and}~ y x_j = x_{j + 1} y \pmod p\rangle.
\]
\end{thm}

 \begin{proof} We suppose that the $p$ singularities of $B$  
divide the circle into intervals $I_1, \ldots, I_p$.  
Note that if $\psi: \G \to \mathbb{Z}_p$ is the map defined in the paragraph preceding the proof of Theorem~\ref{prime}, then $\psi(\G)$ can be only one of two groups: 
the trivial group $\{e\}$ or the group $\mathbb{Z}_p$ and $\G/K \cong \mathcal{S}$ where $S$ is a subgroup of $\mathbb{Z}_p$. 
Thus, $\G/K \cong \{e\}$ or $\G/K \cong \mathbb{Z}_p$. Since we already know $K$, this gives us a description of $\G/K$, but we give a different way to visualize the group below.

We suppose that $\G$ is nontrivial. (The trivial  case can happen if each interval is of type $1$, but the singularities are not all of the
same type, or we have intervals of type $1$ and type $0$.)
 
 \smallskip
{\bf Case 1.} In case $\psi(\G) = \mathbb{Z}_p$, we have $\G/K \cong \mathbb{Z}_p$.  If we choose an element $y \in \G$ that is not mapped to the identity, then $y$ must have order $p$. Therefore, we may assume that $y(I_j) = I_{j + 1}$. Consequently, every interval is of the same type. There are three cases to consider: each interval of type $2$, each of type $1$ and each of type $0$. 

Now, we cannot have every interval of type $0$, for then there are no singularities.

So suppose every interval is type $2$. Then $K$ is nontrivial and we have shifts $x_1, \ldots, x_p$ 
corresponding to intervals $I_j$, for $j = 1, \ldots, p$, and 
$$\G = \langle x_1, x_2 \ldots, x_p, y: y^p = e, x_j x_k = x_k x_j \, \forall j,k,~\mbox{and}~ y x_j = x_{j + 1} y \pmod p\rangle.$$

The remaining case is the case in which every interval is type $1$ (with $B$ having equal one-sided limits). In this case, $K$ is trivial and 
 $$\G = \langle y: y^p = e\rangle \cong \mathbb{Z}_p.$$
 
{\bf Case 2.} In case $\psi(\G) = \{e\}$, there is no nontrivial rotation. Note that because there are singularities, not all intervals are of type $0$.  Thus, we know that at least two intervals are of different types. 
Consequently, we must leave all intervals of type $0$ and type $1$ fixed and we may shift intervals of type $2$ only. 
Thus, if there are $k$ intervals of type $2$ we have $k$ generators that shift the points of the corresponding intervals 
of type $2$, and we have no rotations. Clearly $k < p$.
Therefore $\G \cong \mathbb{Z}^k.$

 \end{proof}

  \section{General case} 

\bigskip

We use the same notation as above for the case in 
which we have $n$ singularities: $K = \{y \in \G: y(I_j) = I_j ~\mbox{for all}~ j\}$. 
From the discussion in the preceding sections it is clear that in the group of invariants, there are only certain possibilities. 
First, an element $x \in \G$ may map every interval to itself. If this is the case, then $x \in K$.
 If $x \notin K$, then $x$ maps some interval $I_j$ to $I_k$ where $j \ne k$. 
As we have seen, $x$ can be written in terms of elements that shift the points of a specific interval and leave all other intervals fixed, 
and elements that move one interval to another.

Recall that 
for $y \in \G$  satisfying $y(I_j) = I_{j + r} \pmod n$ we defined $\psi(y) = [r]$.  Also recall that  $K = \ker \psi$, $K \cong \mathbb{Z}^k$ 
for some $k$ with $0 \le k \le n$ and that $k$ is the number of intervals of type $2$.

We begin with a brief discussion of the case of four singularities, which will help with an understanding of the
general (composite $n$) case.

 \begin{thm} Suppose that the Blaschke product $B$ has exactly 4 singularities on $\TT$. Then the group $\G$ is one of the following possibilities:
$\{e\}$, $\ZZ_2$, $\ZZ_4$, $\ZZ^k$ ($k=1$, $2$ or $3$), $\ZZ^2 \rtimes \ZZ_2$, $\ZZ^4 \rtimes \ZZ_4$.
\end{thm}
 
\begin{proof}
As above, we note that $\G$ has a normal subgroup $K$ isomorphic to $\ZZ^k$ for some $0 \le k \le 4$, with quotient $\G/K$
isomorphic to either
$\{e\}$, $\ZZ_2$ or $\ZZ_4$. Which cases can occur now depends on the types of the intervals $I_1,\ldots,I_4$
comprising $\TT \setminus \sigma(B)$.
 \begin{enumerate}
 \item[Case 1.] All intervals type $2$. Then we have four shifts, $x, y, z, w$, and we can rotate with $r$, where $r^4 = e$. In this case, the group is 
 $$\langle x, y, z, w, r: r^4 = e, xy = yx, \ldots, zw=wz; rw = x r, \ldots, \rangle \cong \ZZ^4\rtimes \ZZ_4.$$
 \item[Case 2.] One interval type $1$, three type $2$. Then $G = \mathbb{Z}^3$ and there is no rotation.
 \item[Case 3.] Three intervals type $1$, one type $2$. Then $G = \mathbb{Z}$ and there is no rotation.
 \item[Case 4.] Two intervals type $2$, two intervals type $1_{a, L}$ (or two intervals type $1_{b, L})$: Then to get a
nontrivial quotient, we must have the type $1$ intervals alternating with the type $2$ intervals, and agreement of the one-sided limits at the points of type $1$.  In this case, we have $G$ described by 
 $$\langle x, y, r: r^2 = e, r x = y r, r y = x r, x y = y x\rangle \cong \ZZ^2 \rtimes \ZZ_2.$$ 
 \end{enumerate}
 The other possibilities (two type $2$ and two type $1$ that are not alternating, and cases including type $0$ intervals) can be handled in the same way.  
 
\end{proof}

We can summarize the general result as follows, noting that some of the cases described below may not occur. 
For example, it is not possible to have $k = 3$ and $d = 2$ in the following theorem.  

 \begin{thm} Suppose that $B$ has exactly $n$ singularities on $\TT$. Then  $\G \cong
\ZZ^k \rtimes \ZZ_d$ for some $0 \le k \le n$ and $1 \le d \le n$ with $d|n$. Here $k$ denotes the
number of intervals of type 2, and $d$ the number of cyclic permutations of the $n$ intervals
comprising $\TT \setminus \sigma(B)$ that map each interval onto another of the same type.
\end{thm}

\begin{proof}
Say that the number of intervals of type $2$ is $k \in \{0, \ldots, n\}$. 
So $K \cong \mathbb{Z}^k$. Consider the map $\psi: \G \to \mathbb{Z}_n$. 
Then $\psi$ maps $\G$ onto a subgroup $\mathcal{S}$ of $\mathbb{Z}_n$. 
Therefore $\G/K \cong \mathcal{S}$, where $\mathcal{S}$ is a subgroup of $\mathbb{Z}_n$; 
in other words, $\G/K$ is $\mathbb{Z}_d$ where $d | n$. Then, letting $x_1, \ldots, x_k$ denote the shifts of the 
corresponding interval of type $2$ that leave all other intervals fixed, we have
\begin{eqnarray*}
\G =  \langle x_j, y: j = 1, \ldots, k; y^d = e, y x_j = x_{j + n/d}y\,  \forall j,  x_j x_l = x_l x_j \,  \forall j, l\rangle\cong
\ZZ^k \rtimes \ZZ_d.\end{eqnarray*}

\end{proof}

{\bf An application:} Suppose that $B$ is a Blaschke product with finitely many singularities $\lambda_1, \lambda_2, \ldots, \lambda_m$ 
and the zeros of $B$ lie in a nontangential region at the point $\lambda_j$ for each $j$, then the group of invariants  is $\ZZ^m \rtimes \ZZ_m$.

\begin{proof} By Lemma~\ref{lemma:factorization} we know that every interval is a type $2$ interval. 
Therefore, letting $x_1, \ldots, x_m$ be the corresponding shifts of each interval and $y$ the map that 
moves $I_j$ to $I_{j + 1}$ in such a way that $y(z_{n, \lambda,j}) = z_{n, \lambda,{j + 1}}$, then 
$$\G = \langle x_j, y: j = 1, \ldots, m; y^m = e,  y x_j  =  x_{j + 1} y, x_j x_l = x_l x_j \forall j \rangle.$$ 
\end{proof}




  
  

\end{document}